\documentclass{amsart}
\usepackage{amssymb}
\usepackage{amscd}
\usepackage{epsfig}

\title{Mixed Pentagon, octagon and Broadhurst duality equation}
\author{Benjamin Enriquez}
\author{Hidekazu Furusho}

\address{Institut de Recherche Math\'{e}matique Avanc\'{e}e, UMR 7501, 
Universit\'{e} de Strasbourg et CNRS, 7
rue Ren\'{e} Descartes, 67000 Strasbourg, France}
\email{enriquez@math.unistra.fr}

\address{Graduate School of Mathematics, Nagoya University, 
Furo-cho, Chikusa-ku, Nagoya, 464-8602, Japan}
\email{furusho@math.nagoya-u.ac.jp}

\newtheorem{thm}{Theorem}[section]
\newtheorem{lem}[thm]{Lemma}
\newtheorem{cor}[thm]{Corollary}
\newtheorem{prop}[thm]{Proposition}  

\theoremstyle{remark}

\theoremstyle{definition}
\newtheorem{thm-defn}[thm]{Theorem-Definition} 
\newtheorem{defn}[thm]{Definition}
\newtheorem{rem}[thm]{Remark}

\newtheorem{q}[thm]{Question}

\numberwithin{equation}{section}
\begin{document}
\bibliographystyle{amsalpha+}
\maketitle
\begin{abstract}
This paper is on elimination of defining equations of
the cyclotomic analogues, introduced by the first author,
of Drinfeld's scheme of associators.
We show that the mixed pentagon equation implies 
the octagon equation for $N=2$ and
the particular distribution relation. 
We also explain that Broadhurst duality is compatible with the torsor structure.
We develop a formalism of infinitesimal module categories and
use it for deriving a proof left implicit in the first named author's earlier work.
\end{abstract}

\tableofcontents

\setcounter{section}{-1}
\section{Introduction}
The ``Grothendieck-Teichm\"{u}ller theory" was developed by Drinfeld
\cite{Dr} with the motivation
of quantization of certain Hopf algebras related with the monodromy of
the KZ differential system
(Kohno-Drinfeld theorem), and in close relation with
Grothendieck's approach to the
description of the action of the
absolute Galois group of the rational number field $\bf Q$
on the``Teichm\"{u}ller tower"
(\cite{Gr}). 
One of the main results of this theory is a collection
of relations between periods of 
$\bold P^1\backslash\{0,1,\infty\}$ called the MZV's (multiple zeta values)
(see \eqref{multiple zeta value}).
These relations are derived from the study of the monodromy of the
KZ system, and fall in three classes : two classes of hexagon
\eqref{hexagons-M} and one
class of pentagon relations \eqref{pentagon-GRT}.
The elimination of the hexagon relations
(i.e., the statement that they are consequences of the pentagon relation) was
established by the second-named author in \cite{F2} (see theorem
\ref{Furusho paper} below). 
The proof uses combinatorial
arguments based on the cell decomposition of the compactification of the
moduli space $\frak M_{0,5}$.


The Grothendieck-Teichm\"{u}ller theory was extended in the cyclotomic context by
the first-named author [E]. In this theory, an integer $N \geqslant 1$ is fixed and
the analogues of the MZV's are periods of 
the motivic fundamental group
the algebraic curve
$\bold P^1\backslash\{0,\mu_N,\infty\}$
($\mu_N$: the group of $N$-th roots of unity)
called multiple $L$-values (see \eqref{multiple L-value}).
The study of the monodromy of the cyclotomic KZ system yields
a collection of relations between these numbers. It is shown that
the pro-algebraic variety $\mathrm{Psdist}(N,\bf k)$  over $\bf k$
($\bf k$: a field of characteristic 0)
arising from the `cyclotomic KZ' relations
is equipped with a torsor structure over pro-$\bf k$-algebraic group  $GRTMD(N,\bf k)$, which is the
extension of $(\bold Z/N\bold Z)^\times \times \bold G_m$
by a prounipotent $\bf k$-algebraic group , and whose Lie
$\bf k$-algebra $\frak{grtmd}(N, \bf k)$ is non-negatively graded.
Another family of relations between MZV's,
the `double shuffle and regularization relations' were studied in \cite{IKZ,R}; cyclotomic analogues of
these relations were discussed in \cite{R}.
In \cite{F3, F4}, it was proved that these
relations are consequences of the `KZ' and `cyclotomic KZ' relations.

For particular values of $N$, exceptional symmetries of
${\bf  P}^1\backslash\{0,\infty,\mu_N\}$ give
rise to additional families of relations. When $N=2$,
these relations were made explicit by Broadhurst (\cite{B}), 
and for N=4, by Okuda (\cite{O}); Okuda's relations
allow one to recover Broadhurst's (see \cite{O} \S 4).
The results of this paper are of four types :
\begin{enumerate}
\renewcommand{\theenumi}{\Alph{enumi}}
\item elimination results for `cyclotomic KZ' relations between multiple $L$-values
\item elimination of defining conditions for pro-algebraic groups and Lie algebras
\item insertion of Broadhurst's result in the framework of torsors
\item a theory of infinitesimal module categories, leading to a proof of
a result left implicit in [E], and which is also used in the proof
of (B).
\end{enumerate}

We now explain these results in more detail.

(A). The `cyclotomic KZ' relations fall in the following classes : the mixed pentagon \eqref{mixed pentagon-GRTM}, octagon \eqref{octagon-Pseudo} and distribution \eqref{distribution} classes ; there is
one class of distribution relations for each $N'$ dividing $N$, $N'\neq N$.
Our first result is the implication of the first distribution relation from the 
mixed pentagon equation:

\begin{thm}[{\bf Proposition \ref{mixed pentagon and distribution}}]
If a pair of two group-like elements $(g,h)\in\exp\frak t^0_3\times\exp\frak t^0_{3,N}$ (for the notations see below)
satisfies the mixed pentagon equation \eqref{mixed pentagon-GRTM},
then it also satisfies the 
distribution relation \eqref{distribution} for $N'=1$.
\end{thm}

As a consequence, we obtain the equality of two groups
$$\mathrm{GRTMD}_{(\bar 1,1)}(N,\bold k)=\mathrm{GRTM}_{(\bar 1,1)}(N,\bold k)$$
and of two torsors
$$\mathrm{Psdist}_{(\bar 1,1)}(N,\bold k)=\mathrm{Pseudo}_{(\bar 1,1)}(N,\bold k)$$
for a prime $N$ (see Corollary \ref{grtmd=grtm}).

(B). The Lie algebras $\frak{grtmd}(N, \bf k)$ are defined by mixed pentagon \eqref{mixed pentagon},
octagon \eqref{octagon}, speciality \eqref{special-b} and distribution \eqref{distribution} equations. 
It was proved in \cite{E} that the speciality condition implies the octagon one. We prove that for $N=2$ the mixed pentagon equation implies the octagon equation:


\begin{thm}[{\bf Theorem \ref{pentagon and octagon in Lie case}}]
For $N=2$ if a pair of two Lie elements
$(\varphi,\psi)\in \frak t^0_{3}\times\frak t^0_{3,2}$
with $c_{B(0)}(\psi)=c_{AB(0)}(\psi)=0$
(for the notations see below)
satisfies the mixed pentagon equation \eqref{mixed pentagon},
then it also satisfies the octagon equation \eqref{octagon}.
\end{thm}


While this result may be viewed as not particularly useful in view of
the result of \cite{E}, its proof might be of interest as it extends the
combinatorial arguments of \cite{F2} to a Kummer covering 
$\tilde{{\frak M}}_{0,5}^2$
of the moduli space $\frak M_{0,5}$.

The pro-algebraic groups $GRTMD(N,\bf k)$ are similarly defined by the
mixed pentagon \eqref{mixed pentagon-GRTM}, octagon \eqref{octagon-GRTM} and speciality equations \eqref{special-b-GRTM}. 
We prove that for $N=2$ 
the octagon equation is implied by the other two equations
in the setting:

\begin{thm}[{\bf Theorem \ref{pentagon and octagon}}]
For $N=2$ if a pair of two group-like element
$(g,h)\in\exp\frak t^0_3\times\exp\frak t^0_{3,2}$ with $c_{B(0)}(h)=c_{AB(0)}(h)=0$
satisfies the mixed pentagon equation \eqref{mixed pentagon-GRTM}
and the special action condition 
\eqref{special-b-GRTM}, 
then it also satisfies the octagon equation \eqref{octagon-GRTM}.
\end{thm}

(C).
In \cite{B},
Broadhurst introduced a family of  \lq duality' relations among multiple $L$-values
for $N=2$.
These relations will be shown to be compatible with the torsor
structure of $\mathrm{Psdist}_{(\bar 1,1)}(2,\bold k)$:

\begin{thm}[{\bf Theorem \ref{Broadhurst torsor}}]
The subset $\mathrm{PseudoB}_{(\bar 1,1)}(2,\bold k)$
defined by the  Broadhurst duality \eqref{duality-GRTM}
forms a subtorsor of $\mathrm{Psdist}_{(\bar 1,1)}(2,\bold k)$.
\end{thm}

(D).
The notion of infinitesimal module categories over
braided monoidal categories is introduced in our appendix.
It is defined  by several axioms including the mixed pentagon axiom.
In Proposition \ref{GRTM forms a group} the notion is employed to prove 
that the set $\mathrm{GRTM}_{(\bar 1,1)}(N,\bf k)$ forms a group by 
the multiplication \eqref{multiplication-GRTM}.

The structure of the paper is the following.
\S\ref{GRT} and \S\ref{cyclotomic GRT} are a review of the
Grothendieck-Teichm\"{u}ller theory in \cite{Dr} and \cite{E}.
In \S\ref{cyclotomic GRT},
elimination result (A) is proved (Proposition \ref{mixed pentagon and distribution}).
Results (B) on mixed
pentagon and octagon relations are proved 
in \S\ref{Mixed pentagon and octagon equation}
(Theorems \ref{pentagon and octagon in Lie case} and \ref{pentagon and octagon}).
Result (C) on compatibility of the Broadhurst duality relations with a torsor
structure is proved in \S\ref{Broadhurst duality} (Theorem \ref{Broadhurst torsor}). Appendix \ref{Infinitesimal module categories} contains result
(D), i.e., the basics of infinitesimal module category and the proof of the
fact implicitly used in \cite{E} that $\mathrm{GRTM}_{(\bar 1,1)}(N,\bf k)$ is a group.
Some errors in  \cite{E} are corrected in
Appendix \ref{Erratum}.

\section{The Grothendieck-Teichm\"{u}ller group}\label{GRT}
This section is a short review on Drinfeld's theory of
associators in \cite{Dr}.

Let $\bold k$ be a field of characteristic 0.
For $n\geqslant 2$, the Lie algebra $\frak t_n$
of infinitesimal pure braids is the completed $\bold k$-Lie algebra
with generators $t^{ij}$ ($i\neq j$, $1\leqslant i,j\leqslant n$)
and relations
$$t^{ij}=t^{ji}, \ [t^{ij},t^{ik}+t^{jk}]=0 \text{ and }
[t^{ij},t^{kl}]=0 \quad  \text{for all distinct $i$, $j$, $k$, $l$.}$$
We note that $\frak t_2$ is the 1-dimensional abelian Lie algebra
generated by $t^{12}$. 
The element $z_{n}=\sum_{1\leqslant i<j\leqslant n}t^{ij}$
is central in $\frak t_n$.
Put $\frak t^0_{n}$ to be the Lie subalgebra of $\frak t_{n}$
with the same generators except $t^{1n}$ and the same relations as $\frak t_{n}$.
Then we have
$$\frak t_{n}=\frak t^0_{n}\oplus\bold k\cdot z_{n}.$$
When $n=3$, $\frak t^0_{3}$ is the free Lie algebra $\frak F_{2}$
of rank $2$ with generators $A:=t^{12}$ and $B:=t^{23}$.

If $S$ and $T$ are two sets, then a
{\it partially defined map} $f:S\to T$ means the data of (a) a subset $D_f \subset S$, and (b) a map $f:D_f\to T$.
For a partially defined map
$f:\{1,\dots,m\}\to\{1,\dots,n\}$,
the Lie algebra morphism 
$\frak t_{n}\to\frak t_{m}$,
$x\mapsto x^f=x^{f^{-1}(1),\dots,f^{-1}(n)}$
is uniquely defined by
$$(t^{ij})^f=\sum_{i'\in f^{-1}(i),j'\in f^{-1}(j)}t^{i'j'}.$$

\begin{defn}[\cite{Dr}]
The Grothendieck-Teichm\"{u}ller
Lie algebra $\frak{grt}_1(\bold k)$ is defined to be the
set of $\varphi=\varphi(A,B)\in \frak t^0_{3}$
satisfying
the {\it duality and hexagon equations} in $\frak t^0_3$
\begin{equation}\label{two hexagons}
\varphi(A,B)+\varphi(B,A)=0, \qquad
\varphi(A,B)+\varphi(B,C)+\varphi(C,A)=0
\end{equation}
with $A+B+C=0$,
the {\it special derivation condition} in $\frak t^0_{3}$
\begin{equation}\label{special}
[B,\varphi(A,B)]+[C,\varphi(A,C)]=0,
\end{equation}
and the {\it pentagon equation} in $\frak t^0_4$
\begin{equation}\label{pentagon}
\varphi^{1,2,34}+\varphi^{12,3,4}=
\varphi^{2,3,4}+\varphi^{1,23,4}+\varphi^{1,2,3}.
\end{equation}
\end{defn}
It actually forms a Lie algebra with the Lie bracket given by
\begin{equation}\label{Lie bracket}
\langle\varphi_1,\varphi_2\rangle=[\varphi_1,\varphi_2]
+D_{\varphi_2}(\varphi_1)-D_{\varphi_1}(\varphi_2),
\end{equation}
where $D_\varphi$ is the derivation of
$\frak t^0_{3}$ defined by
$$D_\varphi(A)=[\varphi,A]\quad \text{and } D_\varphi(B)=0.$$
The Lie algebra structure is realised by the embedding
$$\frak{grt}_1(\bold k)\hookrightarrow \mathrm{Der}(\frak t^0_{3})$$
sending $\varphi\mapsto D_\varphi$.

\begin{defn}[\cite{Dr}]
The Grothendieck-Teichm\"{u}ller group $\mathrm{GRT}_1(\bold k)$ is defined to be the
set of series $g\in \exp\frak t^0_{3}$ satisfying
the {\it duality and hexagon equations} in $\exp\frak t^0_3$
\begin{equation}\label{hexagons-GRT}
g(A,B)g(B,A)=1, \qquad
g(C,A)g(B,C)g(A,B)=1
\end{equation}
with $A+B+C=0$,
the {\it special action condition} in $\exp\frak t^0_{3}$
\begin{equation}\label{special-GRT}
A+g(A,B)^{-1}Bg(A,B)+g(A,C)^{-1}Cg(A,C)=0,
\end{equation}
and the {\it pentagon equation} in $\exp\frak t^0_4$
\begin{equation}\label{pentagon-GRT}
g^{1,2,34}g^{12,3,4}=
g^{2,3,4}g^{1,23,4}g^{1,2,3}.
\end{equation}
\end{defn}

It forms a group by the multiplication
\begin{equation}\label{multiplication-GRT}
g_1\circ g_2=g_2(A,B)\cdot g_1(A,g_2^{-1}Bg_2).
\end{equation}
The group structure is realised by the embedding (but anti-homomorphism)
$$\mathrm{GRT}_1(\bold k)\hookrightarrow \mathrm{Aut}\frak t^0_3$$
sending $g$ to the automorphism $A_g$
defined by 
$$A\mapsto A \ \text{ and } B\mapsto g^{-1}Bg.$$
We note that its associated Lie algebra is $\frak{grt}_1(\bold k)$.

\begin{defn}[\cite{Dr}]
The associator set $\mathrm{M}_1(\bold k)$ is defined to be the
set of series $g\in \exp\frak t^0_{3}$ satisfying
the pentagon equation \eqref{pentagon-GRT} and
the following variant of hexagon equations
\begin{equation}\label{hexagons-M}
g(A,B)g(B,A)=1,\quad
\exp\{\frac{A}{2}\}g(C,A)\exp\{\frac{C}{2}\}g(B,C)\exp\{\frac{B}{2}\}g(A,B)=1
\end{equation}
with $A+B+C=0$.
\end{defn}

It forms a right $\mathrm{GRT}_1(\bold k)$-torsor by \eqref{multiplication-GRT}
with $g_1\in \mathrm{M}_1(\bold k)$ and $g_2\in \mathrm{GRT}_1(\bold k)$.

\begin{rem}\label{GRT_1=M_0}
It is shown that the special derivation condition
\eqref{special} for $\varphi\in\frak t^0_3$
(resp.\eqref{special-GRT} for $g\in\exp \frak t^0_3$)
follows from duality and hexagon equations \eqref{two hexagons}
and the pentagon equation \eqref{pentagon} in \cite{Dr} proposition 5.7.
(resp.\eqref{hexagons-GRT} and \eqref{pentagon-GRT}
in \cite{Dr} proposition 5.9.)
\end{rem}

\begin{rem}
A typical example of elements in $\mathrm{M}_1(\bold C)$
is 
$$\varphi_{KZ}(A,B)=\varPhi_{KZ}
\Bigl(\frac{A}{2\pi\sqrt{-1}},\frac{B}{2\pi\sqrt{-1}}\Bigr)$$
constructed in \cite{Dr},
where $\varPhi_{KZ}(A,B)$ is the {\it Drinfeld associator}.
This series has the following expression:
$$
\varPhi_{KZ}(A,B)=1+
\sum
(-1)^m\zeta(k_1,\cdots,k_m)A^{k_m-1}B\cdots A^{k_1-1}B+
\text{(regularized terms)}
$$
where 
$\zeta(k_1,\cdots,k_m)$ are {\it multiple zeta values} defined by
the following series
\begin{equation}\label{multiple zeta value}
\zeta(k_1,\cdots,k_m)
=\sum_{0<n_1<\cdots<n_m}\frac{1}
{n_1^{k_1}\cdots n_m^{k_m}}
\end{equation}
for $m$, $k_1,\dots, k_m\in {\bf N} (={\bf Z}_{>0})$ with $k_m\neq 1$
and for the regularised terms see \cite{F03}.
\end{rem}

For a monic monomial $W$ in
$U\frak t^0_3={\bf k}\langle\langle A,B\rangle\rangle$,
$c_W(g)$ for $g\in U\frak t^0_3$
means the coefficient of $W$ in $g$.
On pentagon and  hexagon equations we have the following.

\begin{thm}[\cite{F2}]\label{Furusho paper}
(1).
Let $\varphi$ be an element of $\frak t^0_3$
with $c_B(\varphi)=c_{AB}(\varphi)=0$.
If $\varphi$ satisfies the pentagon equation \eqref{pentagon},
then it also satisfies duality and hexagon equations \eqref{two hexagons}.

(2).
Let $g$ be an element of $\exp\frak t^0_3$ with $c_B(g)=c_{AB}(g)=0$.
If $g$ satisfies the pentagon equation \eqref{pentagon-GRT},
then it also satisfies duality and hexagon equations \eqref{hexagons-GRT}.

(3).
Let $g$ be an element of $\exp\frak t^0_3$ with $c_B(g)=0$ and $c_{AB}(g)\in\bold k^\times$.
If $g$ satisfies the pentagon equation \eqref{pentagon-GRT},
then the duality and hexagon equations \eqref{hexagons-M}
hold for $g(\frac{A}{\mu},\frac{B}{\mu})$ with $\mu=\pm\sqrt{24c_{AB}(g)}\in\bar{\bold k}$.
\end{thm}

\begin{rem}
In \cite{F3} it is shown that the pentagon equation \eqref{pentagon-GRT}
implies the double shuffle relation and the regularization relation,
which are one of the fundamental relations among multiple zeta values.
\end{rem}

\section{The cyclotomic Grothendieck-Teichm\"{u}ller group}\label{cyclotomic GRT}
This section is a review of the first named author's
theory on the cyclotomic analogues of associators
in \cite{E}.

Here we recall the notations
\footnote{
Several of them are changed for our convenience.
} in \cite{E}:
For $n\geqslant 2$ and $N\geqslant 1$,
the Lie algebra $\frak t_{n,N}$ is 
the completed $\bold k$-Lie algebra
with generators 
$$t^{1i} \  (2\leqslant i \leqslant n), \text{ and } \ 
t(a)^{ij} \  (i\neq j, 2\leqslant i,j\leqslant n, 
a\in\bold Z/N\bold Z)$$
and relations 
$$t(a)^{ij}=t(-a)^{ji}, \quad$$
$$[t(a)^{ij},t(a+b)^{ik}+t(b)^{jk}]=0,$$
$$[t^{1i}+t^{1j}+\sum_{c\in\bold Z/N\bold Z} t(c)^{ij},t(a)^{ij}]=0,$$
$$[t^{1i},t^{1j}+\sum_{c\in\bold Z/N\bold Z} t(c)^{ij}]=0,$$
$$[t^{1i},t(a)^{jk}]=0 \ 
\text{and } \ 
[t(a)^{ij},t(b)^{kl}]=0$$  
for all $a$, $b\in\bold Z/N\bold Z$ and
all distinct $i$, $j$, $k$, $l$
($2\leqslant i,j,k,l \leqslant n$).\\
We note that $\frak t_{n,1}$ is equal to $\frak t_{n}$ for $n\geqslant 2$.
We have a natural injection $\frak t_{n-1,N}\hookrightarrow\frak t_{n,N}$.
The Lie subalgebra $\frak f_{n,N}$ of $\frak t_{n,N}$ 
generated by $t^{1n}$ and
$t(a)^{in}$ ($2\leqslant  i \leqslant n-1$, $a\in\bold Z/N\bold Z$)
is free of rank $(n-2)N+1$
and forms an ideal of $\frak t_{n,N}$.
Actually it shows that $\frak t_{n,N}$ is a semi-direct product of
$\frak f_{n,N}$ and $\frak t_{n-1,N}$.
The element $z_{n,N}=\sum_{1\leqslant i<j\leqslant n}t^{ij}$
with $t^{ij}=\sum_{a\in\bold Z/N\bold Z}t(a)^{ij}$
($2\leqslant i<j\leqslant n$) is central in $\frak t_{n,N}$.
Put $\frak t^0_{n,N}$ to be the Lie subalgebra of $\frak t_{n,N}$
with the same generators except $t^{1n}$ and the same relations as $\frak t_{n,N}$.
Then we have 
$$\frak t_{n,N}=\frak t^0_{n,N}\oplus\bold k\cdot z_{n,N}.$$
Especially when $n=3$, $\frak t^0_{3,N}$ is free Lie algebra $\frak F_{N+1}$
of rank $N+1$ with generators $A:=t^{12}$ and $B(a)=t(a)^{23}$ ($a\in\bold Z/N\bold Z$).

For a partially defined map 
$f:\{1,\dots,m\}\to\{1,\dots,n\}$
such that $f(1)=1$,
the Lie algebra morphism $\frak t_{n,N}\to\frak t_{m,N}$:
$x\mapsto x^f=x^{f^{-1}(1),\dots,f^{-1}(n)}$ is uniquely defined by
$$(t(a)^{ij})^f=\sum_{i'\in f^{-1}(i),j'\in f^{-1}(j)}t(a)^{i'j'} \quad
(i\neq j, 2\leqslant i,j\leqslant n)$$ 
and
$$
(t^{1j})^f=\sum_{j'\in f^{-1}(j)} t^{1j'}+
\frac{1}{2}\sum_{j',j''\in f^{-1}(j)}\sum_{c\in\bold Z/N\bold Z}t(c)^{j'j''}+
\sum_{i'\neq 1\in f^{-1}(1), j'\in f^{-1}(j)}\sum_{c\in\bold Z/N\bold Z}t(c)^{i'j'}$$
($2\leqslant j\leqslant n$).
Again for a partially defined map 
$g:\{2,\dots,m\}\to\{1,\dots,n\}$,
the Lie algebra morphism $\frak t_{n}\to\frak t_{m,N}$:
$x\mapsto x^g=x^{g^{-1}(1),\dots,g^{-1}(n)}$ is uniquely defined by
$$(t^{ij})^g=\sum_{i'\in g^{-1}(i),j'\in g^{-1}(j)}t(0)^{i'j'}
\qquad
(i\neq j, 1\leqslant i,j\leqslant n).$$

\begin{defn}[\cite{E}]
For $N\geqslant 1$,
the Lie algebra $\frak{grtm}_{(\bar 1,1)}(N,\bold k)$ is defined to be the
set of pairs $(\varphi,\psi)\in \frak t^0_{3}\times\frak t^0_{3,N}$
satisfying $\varphi\in\frak{grt}_1(\bold k)$,\\
the {\it mixed pentagon equation} in $\frak t^0_{4,N}$ 
\begin{equation}\label{mixed pentagon}
\psi^{1,2,34}+\psi^{12,3,4}=
\varphi^{2,3,4}+\psi^{1,23,4}+\psi^{1,2,3},
\end{equation}
the {\it octagon equation} in $\frak t^0_{3,N}$
\begin{align}\label{octagon}
 &\psi\bigl(A,B(0),B(1),\dots,B(i),\dots,B(N-1)\bigr) \\ \notag
-&\psi\bigl(A,B(1),B(2),\dots,B(i+1),\dots,B(0)\bigr)\\ \notag
+&\psi\bigl(C,B(1),B(0),\dots,B(N+1-i),\dots,B(2)\bigr) \\ \notag
-&\psi\bigl(C,B(0),B(N-1),\dots,B(N-i),\dots,B(1)\bigr)=0
\end{align}
with $A+\sum_{a\in\bold Z /N\bold Z}B(a)+C=0$,\\
the {\it special derivation condition} in $\frak t^0_{3,N}$
\begin{align}\label{special-b}
\sum_{a\in\bold Z/N\bold Z}
&\Bigl[\psi\bigl(A,B(a),B(a+1),\dots,B(a+i),\dots,B(a-1)\bigr),B(a)\Bigr] \\ \notag
+\Bigl[\psi\bigl(&A,B(0),B(1),\dots,B(i),\dots,B(N-1)\bigr) \\ \notag 
&-\psi\bigl(C,B(0),B(N-1),\dots,B(N-i),\dots,B(1)\bigr),C\Bigr]
=0
\end{align}
and $c_{B(0)}(\psi)=0$.
\footnote{
For our convenience, we slightly change the original definition
by adding the small condition $c_{B(0)}(\psi)=0$.
The relation to the original Lie algebra is the direct sum decomposition of
Lie algebras ${\frak g}^\text{original}={\frak g}\oplus\bold k\cdot B(0)$,
where ${\frak g}=\frak{grtm}_{(\bar 1,1)}(N,\bold k)$.
}
\end{defn}

Here for any $\bold k$-algebra homomorphism $\iota:U\frak F_{N+1}\to S$
the image $\iota(\varphi)\in S$ is denoted 
by $\varphi(\iota(A),\iota(B(0)),\dots,\iota(B(N-1)))$.
The Lie algebra structure is given by
\begin{equation}
\langle(\varphi_1,\psi_1), (\varphi_2,\psi_2)\rangle
=(\langle\varphi_1,\varphi_2\rangle,
\langle\psi_1,\psi_2\rangle)
\end{equation}
with
$$
\langle\varphi_1,\varphi_2\rangle=[\varphi_1,\varphi_2]
+D_{\varphi_2}(\varphi_1)-D_{\varphi_1}(\varphi_2)
\text{ and }
\langle\psi_1,\psi_2\rangle=[\psi_1,\psi_2]
+\bar{D}_{\psi_2}(\psi_1)-\bar{D}_{\psi_1}(\psi_2).$$
Here $\bar{D}_\psi$ means the derivation of $\frak t^0_{3,N}$
defined by
$$\bar{D}_\psi(A)=[\psi,A], \
\bar{D}_\psi(B(a))=
[\psi-\psi\bigl(A,B(a),B(a+1),\dots,B(a-1)\bigr),B(a)]$$
for $a\in\bold Z/N\bold Z$ and
$$\bar{D}_\psi(C)=
[\psi\bigl(C,B(0),B(N-1),\dots,B(1)\bigr),C].$$
The Lie algebra structure is realised by the embedding
$$\frak{grtm}_{(\bar 1,1)}(N,\bold k)\hookrightarrow
\mathrm{Der}(\frak t^0_{3})\times \mathrm{Der}(\frak t^0_{3,N})$$
sending $(\varphi,\psi)\mapsto (D_\varphi,\bar{D}_\psi)$.

\begin{rem}
It is shown in \cite{E} that
the special derivation condition
\eqref{special-b} for $\psi$ implies 
the octagon equation \eqref{octagon}.
\end{rem}

\begin{defn}[\cite{E}]
For $N\geqslant 1$,
the group $\mathrm{GRTM}_{(\bar 1,1)}(N,\bold k)$ is defined to be the
set of pairs $(g,h)\in \exp\frak t^0_{3}\times\exp\frak t^0_{3,N}$
satisfying $g\in \mathrm{GRT}_1(\bold k)$, $c_{B(0)}(h)=0$,\\
the {\it mixed pentagon equation} in $\exp\frak t^0_{4,N}$
\begin{equation}\label{mixed pentagon-GRTM}
h^{1,2,34}h^{12,3,4}=
g^{2,3,4}h^{1,23,4}h^{1,2,3},
\end{equation}
the {\it octagon equation} in $\exp\frak t^0_{3,N}$
\begin{align}\label{octagon-GRTM}
h\bigl(&A,B(1),B(2),\dots,B(0)\bigr)^{-1} 
h\bigl(C,B(1),B(0),\dots,B(2)\bigr)\cdot  \\ \notag
&h\bigl(C,B(0),B(N-1),\dots,B(1)\bigr)^{-1}
h\bigl(A,B(0),B(1),\dots,B(N-1)\bigr)
=1
\end{align}
with $A+\sum_{a\in\bold Z /N\bold Z}B(a)+C=0$ and \\
the {\it special action condition} in $\exp\frak t^0_{3,N}$
\begin{equation}\label{special-b-GRTM}
A+
\sum_{a\in\bold Z/N\bold Z}
Ad(\tau_ah^{-1}
)(B(a)) 
+Ad\Bigl(h^{-1} 
\cdot h\bigl(C,B(0),B(N-1),\dots,B(1)\bigr)\Bigr)(C)
=0
\end{equation}
where $\tau_a$ ($a\in\bold Z /N\bold Z $) is the automorphism defined by $A\mapsto A$ and
$B(c)\mapsto B(c+a)$ for all $c\in\bold Z /N\bold Z $.
\end{defn}

It forms a group by the multiplication
\begin{align}\label{multiplication-GRTM}
(g_1,h_1)\circ & (g_2,h_2)=\Bigl(g_2(A,B)\cdot g_1(A,Ad(g_2^{-1})(B)),
h_2\bigl(A,B(0),B(1),\dots,B(N-1)\bigr)\cdot \\ \notag
h_1&\bigr(A,Ad(h_2^{-1})B(0), 
Ad(\tau_1h_2^{-1})B(1),
\dots, Ad(\tau_{N-1}h_2^{-1})B(N-1)\bigl)\Bigr).
\end{align}
The group structure is realised by the embedding (but opposite homomorphism)
$$\mathrm{GRTM}_{(\bar 1,1)}(N,\bold k)\hookrightarrow
\mathrm{Aut}\frak t^0_{3}\times \mathrm{Aut}\frak t^0_{3,N}$$
sending $(g,h)$ to the automorphism $(A_g,\bar A_h)$
where $\bar A_h$ is
defined by $A\mapsto A$ and $B(a)\mapsto Ad(\tau_ah^{-1})\left({B(a)}\right)$
for $a\in\bold Z/N\bold Z$.
We note that its associated Lie algebra is $\frak{grtm}_{(\bar 1,1)}(\bold k)$.

\begin{defn}[\cite{E}]
The torsor $\mathrm{Pseudo}_{(\bar 1,1)}(N,\bold k)$ is defined to be the
set of pairs $(g,h)\in \exp\frak t^0_{3}\times\exp\frak t^0_{3,N}$
satisfying $g\in \mathrm{M}_1(\bold k)$, $c_{B(0)}(h)=0$,
the mixed pentagon equation \eqref{mixed pentagon-GRTM}
and the following variant of octagon equation in $\exp\frak t^0_{3,N}$
\begin{align}\label{octagon-Pseudo}
&h\bigl(A,B(1),B(2),\dots,B(0)\bigr)^{-1}\exp\{\frac{B(1)}{2}\}
h\bigl(C,B(1),B(0),\dots,B(2)\bigr)\exp\{\frac{C}{N}\}\cdot  \\ \notag
h\bigl(C,B(0)&,B(N-1),\dots,B(1)\bigr)^{-1}\exp\{\frac{B(0)}{2}\}
\cdot h\bigl(A,B(0),B(1),\dots,B(N-1)\bigr)\exp\{\frac{A}{N}\}
=1.
\end{align}
\end{defn}

It forms a right $\mathrm{GRTM}_{(\bar 1,1)}(N,\bold k)$-torsor 
by \eqref{multiplication-GRTM} with $(g_1,h_1)\in \mathrm{Pseudo}_{(\bar 1,1)}(N,\bold k)$
and $(g_2,h_2)\in \mathrm{GRTM}_{(\bar 1,1)}(N,\bold k)$.

\begin{rem}
In contrast with remark \ref{GRT_1=M_0},
it is not known if \eqref{special-b} and \eqref{special-b-GRTM}
follow from the rest of the equations (cf.\cite{E} remark 7.8).
\end{rem}

\begin{rem}
In \cite{F4} it is shown that the mixed pentagon equation 
\eqref{mixed pentagon-GRTM}
implies the double shuffle relation and the regularization relation
among multiple $L$-values.
\end{rem}

Let $N,N'\geqslant 1$ with $N'|N$. Put $d=N/N'$.
The morphism $\pi_{NN'}:\frak t_{n,N}\to\frak t_{n,N'}$
is defined by
$$t^{1i}\mapsto dt^{1i} \text{ and }
t^{ij}(a)\mapsto t^{ij}(\bar a) \quad
(i\neq j, 2\leqslant i,j\leqslant n, 
a\in\bold Z/N\bold Z),$$
where $\bar a\in\bold Z/N'\bold Z$ means the image of $a$
under the map $\bold Z/N\bold Z\to\bold Z/N'\bold Z$.\\
The morphism $\delta_{NN'}:\frak t_{n,N}\to\frak t_{n,N'}$
is defined by
$$t^{1i}\mapsto t^{1i} \text{ and }
t^{ij}(a)\mapsto 
\begin{cases}
t^{ij}(a/d) \quad\text{ if } \ d|a, \\
t^{ij}(a)\mapsto 0 \text{ if } \ d\nmid a \\
\end{cases}
$$
($i\neq j$, $2\leqslant i,j\leqslant n$, 
$a\in\bold Z/N\bold Z$).
For $\psi\in\frak t^0_{3,N}$,
put $\rho_{NN'}(\psi)=c_{B(0)}(\pi_{NN'}(\psi))-c_{B(0)}(\psi)$.

The morphism $\pi_{NN'}$ (resp. $\delta_{NN'}$)
$:\frak t_{n,N}\to\frak t_{n,N'}$ induces the
morphisms \\
$\frak{grtm}_{(\bar 1,1)}(N,\bold k)\to
\frak{grtm}_{(\bar 1,1)}(N',\bold k)$,
$\mathrm{GRTM}_{(\bar 1,1)}(N,\bold k)\to
\mathrm{GRTM}_{(\bar 1,1)}(N',\bold k)$ and\\
$\mathrm{Pseudo}_{(\bar 1,1)}(N,\bold k)\to
\mathrm{Pseudo}_{(\bar 1,1)}(N',\bold k)$
which we denote  by the same symbol $\pi_{NN'}$ (resp. $\delta_{NN'}$).
We also remark that 
$$\frak{grtm}_{(\bar 1,1)}(1,\bold k)=\frak{grt}_1(\bold k),
\mathrm{GRTM}_{(\bar 1,1)}(1,\bold k)=\mathrm{GRT}_1(\bold k) \text{ and }
\mathrm{Pseudo}_{(\bar 1,1)}(1,\bold k)=\mathrm{M}_1(\bold k).$$

\begin{defn}[\cite{E}]
(1).
For $N\geqslant 1$,
$\frak{grtmd}_{(\bar 1,1)}(N,\bold k)$
is the Lie subalgebra of
$\frak{grtm}_{(\bar 1,1)}(N,\bold k)$ defined
by imposing the {\it distribution relation} in $\frak t^0_{3,N'}$
for all $N'|N$
\begin{equation}\label{distribution}
(\pi_{NN'}-\delta_{NN'})(\psi)=\rho_{NN'}(\psi)B(0).
\end{equation}

(2).
For $N\geqslant 1$, $\mathrm{GRTMD}_{(\bar 1,1)}(N,\bold k)$
is the subgroup of $\mathrm{GRTM}_{(\bar 1,1)}(N,\bold k)$
defined by imposing the {\it distribution relation} in $\exp\frak t^0_{3,N'}$
for all $N'|N$
\begin{equation}\label{distribution-GRTMD}
\pi_{NN'}(h)=e^{\rho_{NN'}(h)B(0)}\delta_{NN'}(h).
\end{equation}

(3).
For $N\geqslant 1$,
the $\mathrm{GRTMD}_{(\bar 1,1)}(N,\bold k)$-torsor $\mathrm{Psdist}_{(\bar 1,1)}(N,\bold k)$
is the subtorsor of
$\mathrm{Pseudo}_{(\bar 1,1)}(N,\bold k)$
defined by imposing the distribution relation \eqref{distribution-GRTMD}
in $\exp\frak t^0_{3,N'}$ for all $N'|N$.
\end{defn}

\begin{rem}
A typical example of  an element of $\mathrm{Psdist}_{(\bar 1,1)}(N,\bold C)$ is 
$$\varphi_{KZ}^N\bigl(A,B(0),\dots,B(N-1)\bigr)=
\varPhi_{KZ}^N\Bigl(\frac{A}{2\pi\sqrt{-1}},\frac{B(0)}{2\pi\sqrt{-1}},\dots,\frac{B(N-1)}{2\pi\sqrt{-1}}\Bigr)$$
where $\varPhi_{KZ}^N(A,B(0),\dots,B(N-1))$ is 
the {\it cyclotomic Drinfeld associator} constructed in \cite{E}.
It has the following expression:
\begin{align*}
\varPhi^N_{KZ}=1+&
\sum (-1)^m L(k_1,\cdots,k_m;
\zeta_1,\dots,\zeta_m)
A^{k_m-1}B(a_m)\cdots A^{k_1-1}B(a_1) \\ \notag
&+\text{(regularized terms)}
\end{align*}
where $\zeta_1=\zeta_N^{a_2-a_1}$, \dots,
$\zeta_{m-1}=\zeta_N^{a_m-a_{m-1}}$,
$\zeta_m=\zeta_N^{-a_m}$ 
with $\zeta_N=\exp\{\frac{2\pi\sqrt{-1}}{N}\}$
and
$L(k_1,\cdots,k_m;\zeta_1,\cdots,\zeta_m)$
are {\it multiple $L$-values} defined by the following series
\begin{equation}\label{multiple L-value}
L(k_1,\cdots,k_m;\zeta_1,\cdots,\zeta_m)
:=\sum_{0<n_1<\cdots<n_m}\frac{\zeta_1^{n_1}\cdots \zeta_m^{n_m}}
{n_1^{k_1}\cdots n_m^{k_m}}
\end{equation}
for $m$, $k_1,\dots, k_m\in {\bf N} (={\bf Z}_{>0})$
and $\zeta_1,\dots,\zeta_m\in\mu_N$
with $(k_m,\zeta_m)\neq (1,1)$.

\end{rem}

The following says that the distribution relation for $N'=1$
follows from the mixed pentagon equation.
Note that the distribution relation for $N'=N$ is automatically satisfied.

\begin{prop}\label{mixed pentagon and distribution}
(1).
Suppose that $(\varphi,\psi)\in\frak t^0_3\times\frak t^0_{3,N}$ satisfies the mixed pentagon equation
\eqref{mixed pentagon} in $\frak t^0_{4,N}$.
Then it also satisfies the distribution relation \eqref{distribution} for $N'=1$
in $\frak t^0_{3,1}=\frak t^0_3$.

(2).
Suppose that $(g,h)\in\exp\frak t^0_3\times\exp\frak t^0_{3,N}$
satisfies the mixed pentagon equation
\eqref{mixed pentagon-GRTM} in $\exp\frak t^0_{4,N}$.
Then it also satisfies the distribution relation \eqref{distribution-GRTMD}
for $N'=1$
in $\exp\frak t^0_{3,1}=\exp\frak t^0_3$.
\end{prop}

\begin{proof}
(1).
By taking the image of \eqref{mixed pentagon} by the composition of $\pi_{N1}$ with 
the projection $\frak t^0_{4,1}=\frak t^0_{4}\to\frak t^0_{3}$ eliminating the first strand,
we get 
$$\pi_{N1}(\psi)=\varphi+Nc_A(\psi)A+c_B(\pi_{N1}(\psi))B.$$
Next by taking the image of \eqref{mixed pentagon} by the composition  of
$\delta_{N1}$ with the projection,
we get 
$$\delta_{N1}(\psi)=\varphi+c_A(\psi)A+c_{B(0)}(\psi)B.$$
By the lemma below
these two equations give \eqref{distribution} for $N'=1$.

(2).
Similarly we obtain
$\pi_{N1}(h)=e^{c_B(\pi_{N1}(h))B}g$ and
$\delta_{N1}(h)=e^{c_{B(0)}(h)B}g$ from \eqref{mixed pentagon-GRTM},
which implies the claim.
\end{proof}

\begin{lem}\label{c=0}
Suppose that $(\varphi,\psi)\in\frak t^0_3\times\frak t^0_{3,N}$ 
(resp. $\in\exp\frak t^0_3\times\exp\frak t^0_{3,N}$)
satisfies the mixed pentagon equation
\eqref{mixed pentagon} in $\frak t^0_{4,N}$
(resp. \eqref{mixed pentagon-GRTM} in $\exp\frak t^0_{4,N}$).
Then $c_A(\psi)=0$.
\end{lem}

\begin{proof}
It can be proved directly by inspecting the terms of degree 1.
\end{proof}

As a corollary, we have

\begin{cor}\label{grtmd=grtm}
For a prime $p$, we have
$$\frak{grtmd}_{(\bar 1,1)}(p,\bold k)=
\frak{grtm}_{(\bar 1,1)}(p,\bold k),$$
$$\mathrm{GRTMD}_{(\bar 1,1)}(p,\bold k)=
\mathrm{GRTM}_{(\bar 1,1)}(p,\bold k),$$
$$\mathrm{Psdist}_{(\bar 1,1)}(p,\bold k)=
\mathrm{Pseudo}_{(\bar 1,1)}(p,\bold k).$$
\end{cor}

\begin{rem}
In \cite{DG} Deligne and Goncharov construct the motivic fundamental group
$\pi_1^{M}(\bold P^1\backslash\{0,1,\mu_N\},1_0)$
($\mu_N$: the group of $N$-th roots of unity)
with the tangential base point $1_0$ at $0$,
which determines a pro-object of the $\bold Q$-linear category
$\mathrm{MT}(\bold Z[\mu_N,\frac{1}{N}])_{\bold Q}$
of mixed Tate motives of $\bold Z[\mu_N,\frac{1}{N}]$.
This causes the morphism 
$$\varphi_N:\mathrm{LieGal}^{M}(\bold Z[\mu_N,\frac{1}{N}])\to \mathrm{Der}\frak t^0_{3,N},$$
where $\mathrm{LieGal}^{M}(\bold Z[\mu_N,\frac{1}{N}])$ is the motivic Lie algebra
of the category.
It is a graded free Lie algebra with 
$\mathrm{rk} K_{2n-1}(\bold Z[\mu_N,\frac{1}{N}])$ generators in each degree $n>0$.
The map $\varphi_N$ is shown to be injective for $N=1$ in \cite{Bw}
and for $N=2,3,4$ and $8$ in \cite{De08}. 
For $N=6$,  a certain modification of the map $\varphi_N$ is 
shown to be injective in \cite{De08}.
Partial injectivity results for $N=2p$ ($p$: a prime)
were obtained in \cite{DW}.
Because all the defining equations of $\frak{grtmd}_{(\bar 1,1)}(N,\bold k)$
are geometric,
it can be shown that 
$\mathrm{Im}\varphi_N$ is embedded in
$\frak{grtmd}_{(\bar 1,1)}(N,\bold k)\subset \mathrm{Der}\frak t^0_{3,N}$.
It is one of the fundamental questions to ask if they are equal or not.
\end{rem}

\section{Mixed pentagon and octagon equations}
\label{Mixed pentagon and octagon equation}
In this section, we focus on the case $N=2$ and
prove that the mixed pentagon equation 
implies the octagon equation.

\begin{thm}\label{pentagon and octagon in Lie case}
Let $(\varphi,\psi)\in \frak t^0_{3}\times\frak t^0_{3,2}$
be a pair satisfying $c_{B(0)}(\psi)=c_{AB(0)}(\psi)=0$ and
the mixed pentagon equation \eqref{mixed pentagon} in $\frak t^0_{4,2}$, i.e.
\begin{align*}
&\psi(t^{12},  t^{23}_++t^{24}_+,t^{23}_-+t^{24}_-)+
\psi(t^{13}+t^{23}_++t^{23}_-,t^{34}_+,t^{34}_-) \\
=\varphi(t^{23}_+,t^{34}_+)+&
\psi(t^{12}+t^{13}+t^{23}_++t^{23}_-,t^{24}_++t^{34}_+,t^{24}_-+t^{34}_-)+
\psi(t^{12},t^{23}_+,t^{23}_-)
\end{align*}
where $t^{ij}_{+}=t^{ij}(0)$ and $t^{ij}_{-}=t^{ij}(1)$.
Then $\psi$ satisfies the octagon equation \eqref{octagon}.
\end{thm}

\begin{proof}
By taking the image of \eqref{mixed pentagon}
by $\delta_{21}$ and eliminating the first strand
we get $\delta_{21}(\psi)=\varphi+c_A(\psi)A$,
which by lemma \ref{c=0} implies $\delta_{21}(\psi)=\varphi$.
Then applying again $\delta_{21}$
to the mixed pentagon equation \eqref{mixed pentagon},
we get the equation \eqref{pentagon} for $\varphi$.
Then by our assumption
$c_B(\varphi)=c_{AB}(\varphi)=0$
and Theorem \ref{Furusho paper} (1),
we have \eqref{two hexagons} for $\varphi$.

For $(\varphi,\psi)\in \frak t^0_{3}\times\frak t^0_{3,2}$,
put 
$$\varPi=\varphi^{2,3,4}+\psi^{1,23,4}+\psi^{1,2,3}
-\psi^{1,2,34}-\psi^{12,3,4}$$
in $\frak t^0_{4,2}$.
Let $\frak S_3$ be the group of permutations of $\{1,2,3,4\}$ 
which fix $\{1\}$.
Then
\begin{align*}
\sum_{\sigma\in \frak S_3}&\epsilon(\sigma)\varPi^{1,\sigma(2),\sigma(3),\sigma(4)}
=(\psi^{1,2,3}-\psi^{1,3,2}) 
+(\psi^{14,3,2}-\psi^{14,2,3})
+(\psi^{13,2,4}-\psi^{13,4,2}) \\
&+(\psi^{12,4,3}-\psi^{12,3,4})
+(\psi^{1,3,4}-\psi^{1,4,3})
+(\psi^{1,4,2}-\psi^{1,2,4})
+\sum_{\sigma\in\frak S_3}\epsilon(\sigma)\varphi^{\sigma(2),\sigma(3),\sigma(4)}.
\end{align*}
There is a unique automorphism $s$ of the Lie algebra $\frak t^0_{4,2}$ such that
$$s(t^{12})=t^{13},
s(t^{13})=t^{12},
s(t^{14})=t^{14},
s(t^{23}_\pm)=t^{23}_\mp,
s(t^{24}_\pm)=t^{34}_\mp \text{ and }
s(t^{34}_\pm)=t^{24}_\pm.$$
Then $s^4=id$ and
$$s(\psi^{12,3,4})=\psi^{13,2,4},
s(\psi^{12,4,3})=\psi^{13,4,2},
s(\psi^{1,4,3})=\psi^{1,4,2},
s(\psi^{1,3,4})=\psi^{1,2,4}.$$
It follows that
\begin{align*}
\sum_{\sigma\in \frak S_3}\epsilon(\sigma)&\varPi^{1,\sigma(2),\sigma(3),\sigma(4)}
=(\psi^{1,2,3}-\psi^{1,3,2}) 
+(\psi^{14,3,2}-\psi^{14,2,3})  \\
&+(s-id)(\psi^{12,3,4}-\psi^{12,4,3}+\psi^{1,4,3}-\psi^{1,3,4})
+\sum_{\sigma\in\frak S_3}\epsilon(\sigma)\varphi^{\sigma(2),\sigma(3),\sigma(4)}.
\end{align*}
Hence

\begin{align*}
(id+s+s^2+s^3)&
\sum_{\sigma\in \frak S_3}\epsilon(\sigma)\varPi^{1,\sigma(2),\sigma(3),\sigma(4)}
=(id+s+s^2+s^3)(\psi^{1,2,3}-\psi^{1,3,2} \\
& +\psi^{14,3,2}-\psi^{14,2,3}) 
+(id+s+s^2+s^3)
\sum_{\sigma\in\frak S_3}\epsilon(\sigma)\varphi^{\sigma(2),\sigma(3),\sigma(4)}.
\end{align*}
By \eqref{mixed pentagon}, $\varPi=0$. 
By \eqref{two hexagons}, the last term is $0$.
So we have
$$
(id+s+s^2+s^3)(\psi^{1,2,3}-\psi^{1,3,2}+\psi^{14,3,2}-\psi^{14,2,3})=0.
$$
Since $s^2(X)=X$ for $X=\psi^{1,2,3},\psi^{1,3,2},\psi^{14,3,2}$ and $\psi^{14,2,3}$,
$$
(id+s)(\psi^{1,2,3}-\psi^{1,3,2}+\psi^{14,3,2}-\psi^{14,2,3})=0.
$$
Let $s'$ be the automorphism of $\frak t^0_{3,2}$ uniquely defined by $s$
sending 
$$s(t^{12})=t^{13}, s(t^{13})=t^{12} \text{ and } s(t^{23}_\pm)=t^{23}_\mp.$$
Then the above equation can be read as
$$\varOmega^{1,2,3}=\varOmega^{14,2,3} \quad \text{ in } \frak t^0_{4,2}$$
where 
$$\varOmega:=\psi^{1,2,3}-\psi^{1,3,2}+s'(\psi)^{1,2,3}-s'(\psi)^{1,3,2}
\in\frak t^0_{3,2}.$$
By the lemma below, $\varOmega$ is described as 
$\varOmega=r(t^{23}_+,t^{23}_-)$ for $r\in\frak F_2$.
So
$$
\psi(t^{12},t^{23}_+,t^{23}_-)-\psi(t^{13},t^{23}_+,t^{23}_-)+
\psi(t^{13},t^{23}_-,t^{23}_+)-\psi(t^{12},t^{23}_-,t^{23}_+)
=r(t^{23}_+,t^{23}_-).
$$
By the identifications
$\frak t^0_{3,2}/(t^{12})\simeq\frak F_2\simeq\frak t^0_{3,2}/(t^{13})$,
we have
$$
\psi(0,t^{23}_+,t^{23}_-)-\psi(-t^{23}_+-t^{23}_-,t^{23}_+,t^{23}_-)+
\psi(-t^{23}_+-t^{23}_-,t^{23}_-,t^{23}_+)-\psi(0,t^{23}_-,t^{23}_+)
=r(t^{23}_+,t^{23}_-),
$$
$$
\psi(-t^{23}_+-t^{23}_-,t^{23}_+,t^{23}_-)-\psi(0,t^{23}_+,t^{23}_-)+
\psi(0,t^{23}_-,t^{23}_+)-\psi(-t^{23}_+-t^{23}_-,t^{23}_-,t^{23}_+)
=r(t^{23}_+,t^{23}_-).
$$
These equalities give $r=0$, which means $\varOmega=0$.
It yields the validity of the octagon equation \eqref{octagon} for $\psi$.
\end{proof}

\begin{lem}
If $X\in\frak t^0_{3,2}$ satisfies $X^{1,2,3}=X^{14,2,3}$ in $\frak t^0_{4,2}$,
then $X$ belongs to the free Lie subalgebra $\frak F_2$ of rank 2 
with generators $t^{23}_+$ and $t^{23}_-$.
\end{lem}

\begin{proof}
Consider the linear map $F:\frak t^0_{3,2}\to\frak t^0_{4,2}$ sending 
$h\mapsto h^{1,2,3}-h^{14,2,3}$.
Its image is contained in the Lie subalgebra of $\frak t^0_{4,2}$
generated by $t^{12}$, $t^{23}_\pm$, $t^{24}_\pm$.
According to \cite{E}, this Lie algebra is freely generated by
these 5 elements.
On the other hand, $\frak t^0_{3,2}$ can be identified with
the free Lie algebra $\frak F_3$ generated by $t^{12}$, $t^{23}_\pm$.
It follows that $\ker F$ is equal to the kernel of the map
$\tilde F:\frak F_3\to\frak F_5$ sending
$$h\mapsto h(t^{12},t^{23}_+,t^{23}_-)-h(t^{12}+t^{24}_++t^{24}_-,t^{23}_+,t^{23}_-).$$
If $X\in\ker\tilde F$, then $X(0,t^{23}_+,t^{23}_-)=X(t^{24}_+,t^{23}_+,t^{23}_-)$,
which implies,
as the Lie subalgebra of $\frak F_5$
generated by $t^{24}_+$, $t^{23}_+$ and $t^{23}_-$
is isomorphic to $\frak F_3$,
that $X$ belongs to the Lie subalgebra  $\frak F_2\subset\frak F_3$
freely generated by $t^{23}_+$ and $t^{23}_-$.
\end{proof}

The following is a geometric interpretation of our arguments above.
\begin{rem}
Put 
$\tilde{\frak{M}}^2_{0,4}:=\{z\in\bold A^1|z\neq 0,\pm 1\}$ and
$$\tilde{\frak{M}}^2_{0,5}:=\{(x,y)\in\bold A^2|xy\neq\pm 1, 
x,y\neq 0, \pm 1\}.$$
These are the the Kummer coverings of the moduli spaces\\
${\frak M}_{0,4}:=\{z\in\bold A^1|z\neq 0,1\}$
with $\tilde{\frak{M}}^2_{0,4}\to{\frak M}_{0,4}:z\mapsto z^2$
and \\
$\tilde{\frak{M}}_{0,5}:=\{(x,y)\in\bold A^2|xy\neq 1, x,y\neq 0,1\}$
with $\tilde{\frak{M}}^2_{0,5}\to{\frak M}_{0,5}:
(x,y)\mapsto (x^2,y^2)$\\
respectively.
The Lie algebras $\frak t^0_{3}$, $\frak t^0_{3,2}$, $\frak t^0_{4}$ and 
$\frak t^0_{4,2}$ are associated with the  fundamental groups
of ${\frak M}_{0,4}$, $\tilde{\frak M}^2_{0,4}$, 
${\frak{M}}_{0,5}$ and $\tilde{\frak{M}}^2_{0,5}$ respectively.
The picture of $\tilde{\frak M}^2_{0,5}$ in Figure 1
is obtained by blowing-ups of $\bold A^2$ at $(x,y)=(0,0),(\pm 1,\pm 1)$ 
and $(\infty,\infty)$.
Our $\varPi$ above corresponds to the pentagon near origin 
surrounded by $\varphi^{2,3,4}$, $\psi^{1,23,4}$, $\psi^{1,2,3}$,
$\psi^{1,2,34}$ and $\psi^{12,3,4}$.
Our $\sum_{\sigma\in \frak S_3}\epsilon(\sigma)\varPi^{1,\sigma(2),\sigma(3),\sigma(4)}$ above
corresponds to the six pentagons in the first quadrant
and
$\sum_{\sigma\in\frak S_3}\epsilon(\sigma)\varphi^{\sigma(2),\sigma(3),\sigma(4)}$
corresponds to the hexagon there.
Our $(id+s+s^2+s^3)
\sum_{\sigma\in \frak S_3}\epsilon(\sigma)\varPi^{1,\sigma(2),\sigma(3),\sigma(4)}$
stands for all the (24-)pentagons in the picture and
$(id+s+s^2+s^3)(\psi^{1,2,3}-\psi^{1,3,2}+\psi^{14,3,2}-\psi^{14,2,3})$
means the two octagons near $(0,0)$ and $(\infty,\infty)$.
\begin{figure}[h]\begin{center}\epsfig{file=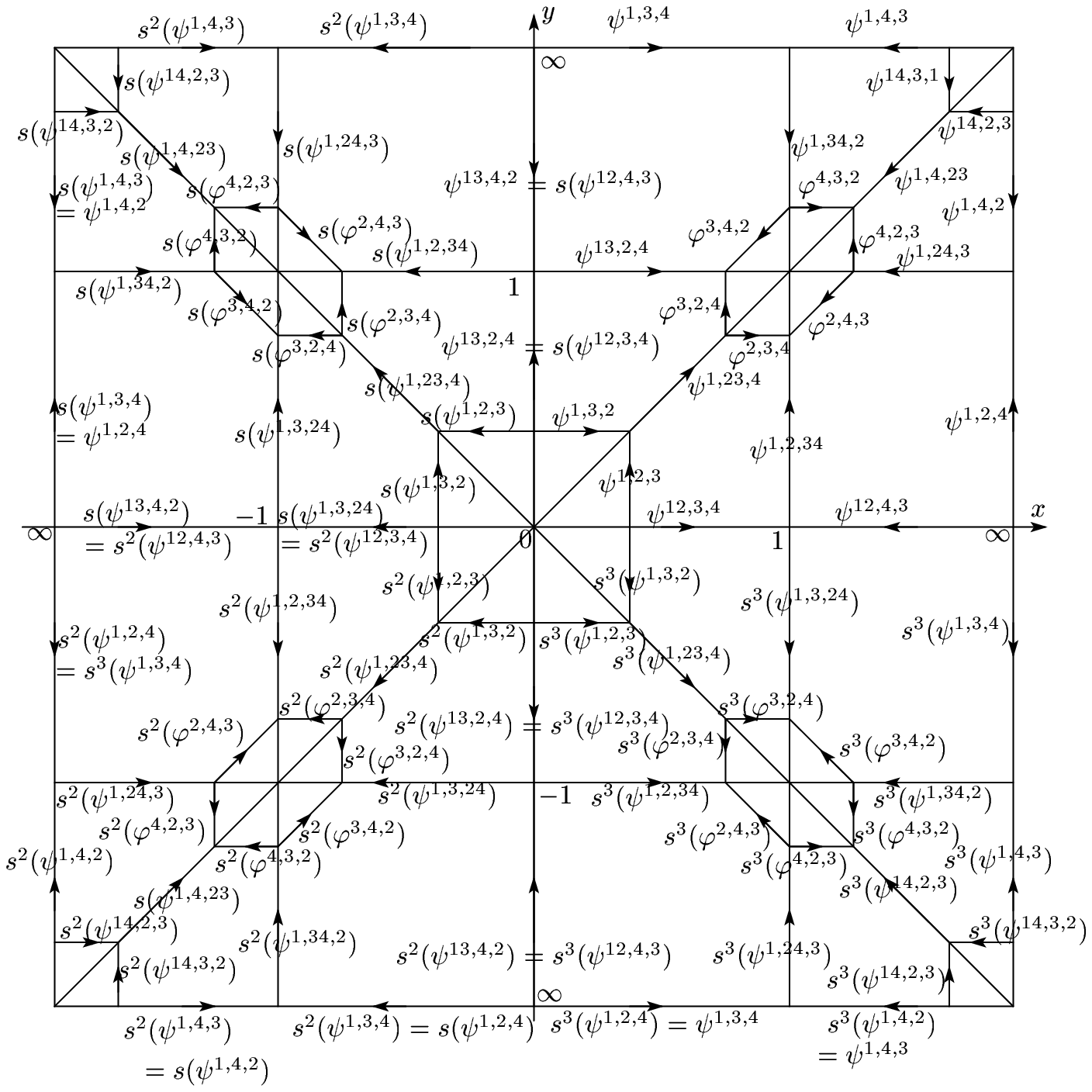, height=360pt}\caption{$\tilde{\frak{M}}^2_{0,5}$}
\end{center}\end{figure}
\end{rem}

Next we show an analogue of Theorem \ref{pentagon and octagon in Lie case}
for group-like series.

\begin{thm}\label{pentagon and octagon}
Let $(g,h)\in\exp\frak t^0_3\times\exp\frak t^0_{3,2}$ be a pair
satisfying $c_{B(0)}(h)=c_{AB(0)}(h)=0$, the mixed pentagon equation \eqref{mixed pentagon-GRTM}
and the special action condition 
\eqref{special-b-GRTM}. 
Then $h$ satisfies the octagon equation \eqref{octagon-GRTM}.
\end{thm}

\begin{proof}
By taking the image of \eqref{mixed pentagon-GRTM} by $\delta_{21}$
and eliminating the first strand,
we get $\delta_{21}(h)=g$ because of Lemma \ref{c=0}
and then the pentagon equation \eqref{pentagon-GRT} for $g$.
By $c_{B(0)}(h)=0$ and \eqref{pentagon-GRT},
the linear terms of $g$ are all zero.
Hence by $c_{AB(0)}(h)=0$, its quadratic terms are all zero.
By Theorem \ref{Furusho paper}.(2), $g\in \mathrm{GRT}_1(\bold k)$.
So it suffices to prove $(g,h)\in \mathrm{GRTM}_{(\bar 1,1)}(2,\bold k)$.
The proof can be done by induction on degree.
Suppose that we have \eqref{octagon-GRTM} for $(g,h)$ modulo degree $n$,
which we denote as 
$$(g,h)\pmod{\deg n}\in \mathrm{GRTM}_{(\bar 1,1)}(2,\bold k)^{(n)}.$$
Then there is a pair 
$$(g_1,h_1)\in \mathrm{GRTM}_{(\bar 1,1)}(2,\bold k) \ 
\text{with} \  (g,h)\equiv(g_1,h_1)\pmod{\deg n}$$ 
by Lemma \ref{surjection}.
Let $(g_0,h_0)$ be the pair defined by 
$$(g,h)=(g_0,h_0)\circ(g_1,h_1).$$
Then the pair $(g_0,h_0)$ lies in $\exp\frak t^0_3\times\exp\frak t^0_{3,2}$
and satisfies $c_{B(0)}(h_0)=c_{AB(0)}(h_0)=0$,
\eqref{mixed pentagon-GRTM} and \eqref{special-b-GRTM}
by $(g_1,h_1)\in \mathrm{GRTM}_{(\bar 1,1)}(2,\bold k)$.
By $(g,h)\equiv(g_1,h_1)\pmod{\deg n}$, we have $(g_0,h_0)\equiv (1,1)\pmod{\deg n}$.
Denote the degree $n$-part of the pair $(g_0,h_0)$ by $(\varphi,\psi)$.
The pair $(\varphi,\psi)$ lies in $\frak t^0_3\times\frak t^0_{3,2}$
and satisfies \eqref{mixed pentagon} and \eqref{special-b}
by \eqref{mixed pentagon-GRTM} and \eqref{special-b-GRTM} for 
$(g,h)$ and $(g_0,h_0)$,
which is obtained by comparing the lowest differing terms of the equations.
Then by Theorem \ref{pentagon and octagon in Lie case},
we have 
$$(\varphi,\psi)\in\frak{grtm}_{(\bar 1,1)}(2,\bold k).$$
Let $(g'_0,h'_0)$ be the element in $\mathrm{GRTM}_{(\bar 1,1)}(2,\bold k)$
which corresponds to $(\varphi,\psi)\in\frak{grtm}_{(\bar 1,1)}(2,\bold k)$
by the exponential map.
Since $(g_0,h_0)\equiv(g'_0,h'_0)\pmod{\deg n+1}$,
$$(g_0,h_0)\pmod{\deg n+1}\in \mathrm{GRTM}_{(\bar 1,1)}(2,\bold k)^{(n+1)}.$$
Therefore $(g,h)\pmod{\deg n+1}\in \mathrm{GRTM}_{(\bar 1,1)}(2,\bold k)^{(n+1)}$.
\end{proof}

\begin{rem}
We note that
in 
Theorem \ref{pentagon and octagon in Lie case}
we do not assume  the special condition \eqref{special-b},
on the other hand, in 
Theorem \ref{pentagon and octagon}
we assume the special condition \eqref{special-b-GRTM}.
The analogue of Theorem \ref{Furusho paper} (3)
might be the implication of \eqref{octagon-Pseudo}
from \eqref{mixed pentagon}
but we do not know whether this implication holds.
\end{rem}

\section{Broadhurst duality}
\label{Broadhurst duality}
We will show that the Broadhurst duality relation is compatible with
the torsor structure of $\mathrm{Pseudo}_{(\bar 1,1)}(2,\bold k)$.

Let $\tau$ be the involution of $\frak t^0_{3,2}$
defined by $\tau:A\leftrightarrow B(0)$ and $B(1)\leftrightarrow C$.

\begin{defn}
(1).
The set $\frak{grtmb}_{(\bar 1,1)}(2,\bold k)$
is defined as the set of all $\psi\in\frak{grtm}_{(\bar 1,1)}(2,\bold k)$
such that the {\it Broadhurst duality relation}
\begin{equation}\label{duality}
\tau(\psi)+\psi+\alpha_\psi(A+B(0))=0
\end{equation}
holds for some $\alpha_\psi\in \bold k$.

(2).
The set $\mathrm{GRTMB}_{(\bar 1,1)}(2,\bold k)$
is defined as the set of all $(g,h)\in\mathrm{GRTM}_{(\bar 1,1)}(2,\bold k)$
such that the {\it Broadhurst duality relation}
\begin{equation}\label{duality-GRTM}
\tau(h)e^{\alpha_h B(0)}he^{\alpha_h A}=1
\end{equation}
holds for some $\alpha_h\in \bold k$.

(3).
The set $\mathrm{PseudoB}_{(\bar 1,1)}(2,\bold k)$
is defined as the set of all
$(g,h)\in\mathrm{Pseudo}_{(\bar 1,1)}(2,\bold k)$
such that the {\it Broadhurst duality relation} \eqref{duality-GRTM}
holds for some $\alpha_h\in \bold k$.
\end{defn}

Actually $\alpha_\psi$ and $\alpha_h$ are equal to the coefficients of $B(1)$ in $\psi$ and $h$ respectively.

\begin{thm}\label{Broadhurst torsor}
(1).
The set $\frak{grtmb}_{(\bar 1,1)}(2,\bold k)$
forms a Lie algebra by the Lie bracket \eqref{Lie bracket}.

(2).
The set $\mathrm{GRTMB}_{(\bar 1,1)}(2,\bold k)$
forms an algebraic group by the multiplication \eqref{multiplication-GRT}
and its associated Lie algebra is  $\frak{grtmb}_{(\bar 1,1)}(2,\bold k)$.

(3).
The set $\mathrm{PseudoB}_{(\bar 1,1)}(2,\bold k)$ forms a right
$\mathrm{GRTMB}_{(\bar 1,1)}(2,\bold k)$-torsor by \eqref{multiplication-GRT}.
\end{thm}

\begin{proof}
(1).
Put $\mathrm{Out Der}(\frak t^0_{3,2})=\mathrm{(Der/Int)}(\frak t^0_{3,2})$.
This quotient forms a Lie algebra
with the involution induced by $\tau$.
Its invariant part  $\mathrm{Out Der}^+(\frak t^0_{3,2})$ again forms a Lie algebra.
The embedding sending $(\varphi,\psi)\mapsto (D_\varphi,\bar{D}_\psi)$
induces the embedding
$\frak{grtm}_{(\bar 1,1)}(2,\bold k)\hookrightarrow
\mathrm{Der}(\frak t^0_{3})\times \mathrm{OutDer}(\frak t^0_{3,2})$.
It can be checked that \eqref{duality} is
the condition for $(\varphi,\psi)$ to belong to
the intersection of two Lie algebras
$\frak{grtm}_{(\bar 1,1)}(2,\bold k)$
and $\mathrm{Der}(\frak t^0_{3})\times \mathrm{OutDer}^+(\frak t^0_{3,2})$.

(2).
It can be proved similarly.
Put $\mathrm{Out} \frak t^0_{3,2}=\mathrm{(Aut/Inn)}(\frak t^0_{3,2})$,
the outer automorphism group of $\frak t^0_{3,2}$;
the group of automorphisms modulo inner automorphisms.
This quotient forms a group
with the involution induced from $\tau$.
Its invariant part $\mathrm{Out}^+\frak t^0_{3,2}$ again forms a group.
The embedding sending $(g,h)\mapsto (A_g,\bar{A}_h)$
induces the embedding
$\mathrm{GRTM}_{(\bar 1,1)}(2,\bold k)\hookrightarrow
\mathrm{Aut}(\frak t^0_{3})\times \mathrm{Out}(\frak t^0_{3,2})$.
It can be checked that \eqref{duality-GRTM} is
the condition for $(g,h)$ to belong to
the intersection of two group
$\mathrm{GRTM}_{(\bar 1,1)}(2,\bold k)$
and 
$\mathrm{Aut}(\frak t^0_{3})\times \mathrm{Out}^+(\frak t^0_{3,2})$.
Since $\frak{grtm}_{(\bar 1,1)}(2,\bold k)$
and $\mathrm{Der}(\frak t^0_{3})\times \mathrm{OutDer}^+(\frak t^0_{3,2})$ are
the associated Lie algebras with these two groups,
$\frak{grtmb}_{(\bar 1,1)}(2,\bold k)$ is associated with
$\mathrm{GRTMB}_{(\bar 1,1)}(2,\bold k)$.

(3).
By direct calculation, it can be shown that
$$\tau(h_3)e^{(\alpha_{h_1}+\alpha_{h_2})B(0)}h_3
e^{(\alpha_{h_1}+\alpha_{h_2})A}=1$$
for $(g_3,h_3)=(g_1,h_1)\circ (g_2,h_2)$ with
$(g_1,h_1)\in \mathrm{PseudoB}_{(\bar 1,1)}(2,\bold k)$
and
$(g_2,h_2)\in \mathrm{GRTMB}_{(\bar 1,1)}(2,\bold k)$,
which shows that $\mathrm{PseudoB}_{(\bar 1,1)}(2,\bold k)$
is a $\mathrm{GRTMB}_{(\bar 1,1)}(2,\bold k)$-space.
To prove that it forms a torsor, it suffices to show that the action is transitive.
Assume that $(g_1,h_1)$ and $(g_3,h_3)$ belong to 
$\mathrm{PseudoB}_{(\bar 1,1)}(2,\bold k)$
and they are equal $\mod\deg n-1$.
Then the degree $n$-part $\psi$ of their difference satisfies $\tau(\psi)+\psi=0$.
So it gives an element $(\varphi,\psi)\in\frak{grtmb}_{(\bar 1,1)}(2,\bold k)$.
Put 
$$(g^{(n)}_2,h^{(n)}_2):=Exp(\varphi,\psi)\in\mathrm{GRTMB}_{(\bar 1,1)}(2,\bold k).$$
Let $(g_2,h_2)\in \mathrm{GRTM}_{(\bar 1,1)}(2,\bold k)$
be the element uniquely determined by $(g_3,h_3)=(g_1,h_1)\circ (g_2,h_2)$.
Then $(g_2,h_2)\equiv(g^{(n)}_2,h^{(n)}_2)\mod\deg n$.
By approximation methods replacing $(g_1,h_1)$ by $(g_1,h_1)\circ(g^{(n)}_2,h^{(n)}_2)$,
we can show $(g_2,h_2)$ belongs to
$\mathrm{GRTMB}_{(\bar 1,1)}(2,\bold k)$.
\end{proof}

We note that by Corollary \ref{grtmd=grtm}
$$\frak{grtmdb}_{(\bar 1,1)}(2,\bold k)=\frak{grtmb}_{(\bar 1,1)}(2,\bold k),$$ 
$$\mathrm{GRTMDB}_{(\bar 1,1)}(2,\bold k)=\mathrm{GRTMB}_{(\bar 1,1)}(2,\bold k),$$
$$\mathrm{PsdistB}_{(\bar 1,1)}(2,\bold k)=\mathrm{PseudoB}_{(\bar 1,1)}(2,\bold k).$$

\begin{rem}
(i).
The equation \eqref{duality-GRTM} holds for  $h=\varPhi^{N}_{KZ}$ with 
$\alpha=\log 2$ and $N=2$ (cf. \cite{O} \S 4).
Here $\varPhi^{N}_{KZ}$ means an $N$-cyclotomic analogue of the Drinfeld associator,
whose all coefficients are multiple $L$-values (see also \cite{E}).
It is explained in \cite{O} that
the equation yields the Broadhurst duality relation \cite{B} (127)
of multiple $L$-values with signature $\{\pm\}$.

(ii).
In \cite{LNS}, a subgroup  $\mathrm{I}\!\Gamma$ of
the pro-finite Grothendieck-Teichm\"{u}ller group $\widehat{GT}$
is introduced, with the properties of both containing 
the absolute Galois group $Gal(\overline{\bf Q}/{\bf Q})$
of the rational number field $\bf Q$
and of acting on the pro-finite completion of all the mapping class
groups.
One of its main defining conditions is Equation (IV) from \cite{LNS},
an equivalent form of which was found in \cite{F1} Equation (3).
The latter equation is a pro-finite analogue of \eqref{duality-GRTM}.
\end{rem}

Since \eqref{duality} is geometric,
$\frak{grtmdb}_{(\bar 1,1)}(2,\bold k)$ contains the free Lie algebra
$\mathrm{LieGal}^{M}(\bold Z[\frac{1}{2}])$
with one free generator in each degree $1$, $3$, $5$, $7$,\dots.
It is fundamental to ask if they are equal or not. Namely
\begin{q}
Is the Lie algebra $\frak{grtmdb}_{(\bar 1,1)}(2,\bold k)$
free with one generator in each degree $1$, $3$, $5$, $7$,\dots ?
\end{q}

\appendix
\section{Infinitesimal module categories}\label{Infinitesimal module categories}
In this appendix basics of infinitesimal module categories
are given.
We prove the fact implicitly employed in \cite{E} that
$\mathrm{GRTM}_{(\bar 1,1)}(N,\bf k)$ forms a group
by using its action on infinitesimal module categories.

\subsection{Infinitesimal braided monoidal categories}
An {\it infinitesimal braided monoidal category} (IBMC for short)
is a set 
$$\mathbb C=(\mathcal C,\otimes,I,a,c,l,r,U,t)$$
consisting of a category $\mathcal C$, 
a bi-functor $\otimes:\mathcal C^2\to\mathcal C$,
$I\in\text{Ob}\mathcal C$,
functorial assignments
$a_{XYZ}\in \mathrm{Isom}_{\mathcal C}(X\otimes (Y\otimes Z),(X\otimes Y) \otimes Z)$ and
$c_{XY}\in \mathrm{Isom}_{\mathcal C}(X\otimes Y, Y\otimes X)$,
$l_X\in \mathrm{Isom}_{\mathcal C}(I\otimes X,X)$,
$r_X\in \mathrm{Isom}_{\mathcal C}(X\otimes I,X)$,
a normal subgroup $U_X$ of $\mathrm{Aut}_{\mathcal C}X$
and $t_{XY}\in \mathrm{Lie} U_{X\otimes Y}$
for all $X,Y,Z\in\text{Ob}\mathcal C$
which satisfies the following:

(i). It forms a {\it braided monoidal category} \cite{M}
(quasi-tensor category \cite{Dr}):
{\it the triangle, the pentagon and
the hexagon axioms} hold for $a,c,l,r$ and $I$
(cf. \cite{Dr} (1.7)-(1.9b)).

(ii). The group $U_X$ is a pro-unipotent $\bf k$-algebraic group
and $fU_Xf^{-1}=U_Y$ for any $f\in \mathrm{Isom}_{\mathcal C}(X,Y)$ holds
and any $X,Y\in \mathrm{Ob}\mathcal C$.

(iii). The map $t_{XY}$ is functorial on $X$ and $Y$ and satisfies
$$t_{X\otimes Y, Z}=
a_{XYZ}(id_X\otimes t_{YZ})a_{XYZ}^{-1}+
(c_{YX}\otimes id_Z)a_{YXZ}(id_Y\otimes t_{XZ})
a_{YXZ}^{-1}(c_{YX}\otimes id_Z)^{-1}$$
and
$c_{XY}t_{XY}=t_{YX}c_{XY}$.

For a group $G$ and a (braided) monoidal category $\mathcal C$,
{\it an action of $G$ on $\mathcal C$} is a collection of morphisms
$G\to \mathrm{Aut}_{\mathcal C}X:g\mapsto g_X$
for $X\in \mathrm{Ob}\mathcal C$ such that
$fg_X=g_Yf$ for any $f\in \mathrm{Isom}_{\mathcal C}(X,Y)$ and
$g_{X\otimes Y}=g_X\otimes g_Y$ for any $g\in G$ and $X,Y\in \mathrm{Ob}\mathcal C$.
Let $C_N$ be the cyclic group $\bold Z/N$ of order $N\in\bold N$ with a generator $\sigma$.
We call an IBMC with $C_N$-action a $C_N$-IBMC.

For $n\geqslant 1$, let $\frak u_{n,N}$ denote be the completed $\bf k$-Lie algebra
with generators 
$$t(a)^{ij} \quad  (i\neq j,1\leqslant i,j\leqslant n,
a\in C_N )$$ 
and relations 
$$t(a)^{ij}=t(-a)^{ji},
[t(a)^{ij},t(a+b)^{ik}+t(b)^{jk}]=0 \text{ and }
[t(a)^{ij},t(b)^{kl}]=0$$
for all $a,b\in C_N$ and all distinct $i,j,k,l$
($1\leqslant i,j,k,l\leqslant n$).
Denote by $G_{n,N}$ the semi-direct product of
$C_N^n$ by the symmetric group $S_n$.
This group acts on ${\frak u}_{n,N}$ by
$$(c_1,\cdots,c_n)\cdot t(a)^{ij}=t(a+c_i-c_j) \text{ and }
s\cdot t(a)^{ij}=t(a)^{s(i),s(j)}$$ 
for
$1\leqslant i,j\leqslant n$,
$a,c_1\dots,c_n\in C_N$ and $s\in S_n$.
Put 
$\mathcal U_{n,N}=\exp{\frak u}_{n,N}$ and
$\mathcal G_{n,N}=\mathcal U_{n,N}\rtimes G_{n,N}$.
Let $\mathbb C_{\text{univ}}$
be the category defined by
$$\text{Ob}\mathbb C_{\text{univ}}=\coprod_{n\geqslant 0}\{\text{parenthesizations of the word }
\underbrace{\bullet\cdots\bullet}_{n}\}$$
and for $X,X'\in\text{Ob}\mathbb C_{\text{univ}}$,
$$\mathrm{Mor}_{\mathbb C_{\text{univ}}}(X,X')=
\begin{cases}
\mathcal G_{n,N} &
\text{if their lengths $|X|$ and $|X'|$ are equal to $n$},\\
\emptyset 
&\text{if their lengths are different.}
\end{cases}
$$
Let $\otimes:(\text{Ob}\mathbb C_{\text{univ}})^2\to
\text{Ob}\mathbb C_{\text{univ}}$ be the map induced by the concatenation
and $\mathcal G_{m,N}\times\mathcal G_{n,N}\to\mathcal G_{m+n,N}$
be the homomorphism induced by the juxtaposition.
They yield a morphism
$\mathrm{Mor}(X,X')\times \mathrm{Mor}(Y,Y')\to \mathrm{Mor}(X\otimes Y,X'\otimes Y')$.
Put $a_{XYZ}:=1\in\mathcal G_{|X|+|Y|+|Z|,N}$ and
$$c_{XY}:=\sigma_{|X|,|Y|}\in S_{|X|+|Y|}\subset G_{|X|+|Y|,N}\subset \mathcal G_{|X|+|Y|,N}$$
where $\sigma_{|X|,|Y|}$ means the permutation interchanging $X$ and $Y$.
Set  $I=\emptyset \ \text{(:the empty word)}\in\text{Ob}\mathbb C_{\text{univ}}$,
$l_X$ and $r_X$ to be the identity maps. 
Finally we put $U_X=\mathcal U_{m,N}\in \mathrm{Aut}_{\mathbb C_{\text{univ}}}(X)$
and 
$$t_{XY}:=\sum_{1\leqslant i\leqslant m,1\leqslant j\leqslant n,a\in C_N}
t(a)^{i,m+j}\in {\frak u}_{m+n,N}$$
for $|X|=m$ and $|Y|=n$.
Then $\mathbb C_{\text{univ}}$ forms a $C_N$-IBMC, which is universal in the following sense:
if $\mathbb C$ is a $C_N$-IBMC with a distinguished object $X$,
then there exists a unique functor $\mathbb C_{\text{univ}}\to\mathbb C$
of $C_N$-IBMC's which sends $\bullet$ to $X$.

\subsection{Infinitesimal module categories over $C_N$-braided monoidal categories}
Let $\mathbb C=(\mathcal C,\otimes,I,a,c,l,r,U,t,\sigma)$ be a $C_N$-IBMC.
An {\it infinitesimal (right) module category} (IMC for short) over $\mathbb C$
is a set 
$$\mathbb M=(\mathcal M,\otimes,b,r,V,t)$$
consisting of a category $\mathcal M$, 
a bi-functor $\otimes:\mathcal M\otimes\mathcal C\to\mathcal M$,
functorial assignments
$b_{MXY}\in \mathrm{Isom}_{\mathcal M}(M\otimes (X\otimes Y),(M\otimes X) \otimes Y)$ and
$r_M\in \mathrm{Isom}_{\mathcal M}(M\otimes I, M)$,
a normal subgroup $V_M$ of $\mathrm{Aut}_{\mathcal M}M$
and $t_{MX}\in \mathrm{Lie} V_{M\otimes X}$
for all $M\in\text{Ob}\mathcal M$ and $X,Y\in\text{Ob}\mathcal C$ which satisfies the following:

(I). It forms a {\it right module category} over $(\mathcal C,\otimes,I,a,c,l,r)$:
the {\it mixed pentagon axiom}
$$(b_{MXY}\otimes id_Z)b_{M,X\otimes Y,Z}=
b_{M\otimes X,Y,Z}b_{M,X,Y\otimes Z}(id_M\otimes a_{XYZ})$$
and the {\it triangle axioms}
$$r_{M\otimes X}b_{MXI}=id_M\otimes r_X \quad \text{and} \quad
(r_M\otimes id_X)b_{MIX}=id_M\otimes l_X$$ hold
for all $M\in\text{Ob}\mathcal M$ and $X,Y,Z\in\text{Ob}\mathcal C$.

(II). The {\it octagon axiom}
$$id_{M\otimes X}\otimes\sigma_Y=b_{MXY}(id_M\otimes c_{Y\otimes X})
b_{MYX}^{-1}((id_M\otimes \sigma_Y)\otimes id_X)
b_{MYX}^{-1}\cdot(id_M\otimes c_{XY})b_{MXY}$$
holds
for all $M\in\text{Ob}\mathcal M$ and $X,Y\in\text{Ob}\mathcal C$.

(III). The group $V_M$ is a pro-unipotent $\bf k$-algebraic group
and $fV_Mf^{-1}=V_{M'}$ holds for any
$f\in \mathrm{Isom}_{\mathcal M}(M,M')$ and any $M,M'\in \mathrm{Ob}\mathcal M$.

(IV). The map $t_{MX}$ is functorial on $M$ and $X$ and satisfies
\begin{align*}
t_{M\otimes X,Y}=b_{MXY}&(id_M\otimes c_{YX})b_{MYX}^{-1}\cdot
(t_{MY}\otimes id_X)b_{MYX}(id_M\otimes c_{XY})b_{MXY}^{-1} \\
&+\sum_{a\in C_N}(id_{M\otimes X}\otimes\sigma_Y^a)\cdot
b_{MXY}(id_M\otimes t_{XY})b_{MXY}^{-1}\cdot
(id_{M\otimes X}\otimes\sigma_Y^{-a})
\end{align*}
and
$$t_{M\otimes X,Y}+t_{MX}\otimes id_Y=b_{MXY}t_{M,X\otimes Y}b_{MXY}^{-1}.$$

We can formulate the notion of functors between 
two IMC's over $C_N$-IBMC's.
We note that such category
forms a braided module category in the sense of \cite{E}.

A natural morphism $\frak u_{n,N}\to\frak t_{n+1,N}$
is obtained by shifting indices by 1.
By the morphism we extend the $G_{n,N}$-action on $\frak u_{n,N}$
into on $\frak t_{n+1,N}$ via
$$(c_1,\cdots,c_n)\cdot t^{1,i+1}=t^{1,i+1} \ \text{ and } \ 
\sigma(t^{1,i+1})=t^{1,\sigma(i)+1}$$
for $c_1,\cdots,c_n\in C_N$, $\sigma\in S_n$ and $1\leqslant i\leqslant n$.
Put $\tilde{\mathcal U}_{n+1,N}:=\exp\frak t_{n+1,N}$ and
$\tilde{\mathcal G}_{n+1,N}:=\tilde{\mathcal U}_{n+1,N}\rtimes G_{n,N}$.
We now construct an IMC $\mathbb M_{\text{univ}}$ over
the $C_N$-IBMC $\mathbb C_{\text{univ}}$.
Set 
$$\text{Ob}\mathbb M_{\text{univ}}=\text{Ob}\mathbb C_{\text{univ}}$$
and for $M,M'\in\text{Ob}\mathbb M_{\text{univ}}$,
$$
\mathrm{Mor}_{\mathbb M_{\text{univ}}}(M,M')=
\begin{cases}
\tilde{\mathcal G}_{n+1,N}
&\text{if their lengths $|M|$ and $|M'|$ are equal to $n$,} \\
\emptyset 
&\text{if their lengths are different.}
\end{cases}
$$
Let $\otimes:\text{Ob}\mathbb M_{\text{univ}}\times
\text{Ob}\mathbb C_{\text{univ}}\to
\text{Ob}\mathbb C_{\text{univ}}$ be the map induced by the concatenation
and $\tilde{\mathcal G}_{m+1,N}\times\mathcal G_{n,N}
\to\tilde{\mathcal G}_{m+n+1,N}$
be the homomorphism induced by the juxtaposition.
They yield a morphism
$\mathrm{Mor}(M,M')\times \mathrm{Mor}(X,X')\to \mathrm{Mor}(M\otimes X,M'\otimes X')$.
Put $b_{MXY}:=1\in\mathcal G_{|M|+|X|+|Y|+1,N}$ and $r_M:=id_M$.
Finally we put $V_M=\tilde{\mathcal U}_{m+1,N}
\subset\tilde{\mathcal G}_{m+1,N}=\mathrm{Aut}_{\mathbb M_{\text{univ}}}(M)$
and 
$$t_{MX}:=\sum_{1\leqslant j\leqslant n}
\sum_{0\leqslant i\leqslant m+j-1}t^{i+1,j+m+1}
\in {\frak t}_{m+n+1,N}$$
for $|M|=m$ and $|X|=n$.
Then it can be shown that $\mathbb M_{\text{univ}}$ forms an IMC over
the $C_N$-IBMC $\mathbb C_{\text{univ}}$,
which is universal in the following sense:
if $\mathbb M$ is an IMC over a $C_N$-IBMC $\mathbb C$
with distinguished objects 
$M\in\text{Ob}\mathbb M$ and $X\in\text{Ob}\mathbb C$,
then there exists unique functors $\mathbb M_{\text{univ}}\to\mathbb M$
and $\mathbb C_{\text{univ}}\to\mathbb C$
of IMC's over $C_N$-IBMC's which send $\bullet$ to $M$ and $X$ respectively.

\subsection{Automorphisms}
Let $\mathbb C=(\mathcal C,\otimes,I,a,c,l,r,U,t)$ be an IBMC.
Let $g\in\exp\frak t^0_3$.
Set 
$$\tilde a_{XYZ}:=a_{XYZ}\cdot g(a_{XYZ}^{-1}(t_{XY}\otimes id_Z)a_{XYZ},
id_X\otimes t_{YZ})^{-1}.$$
Then the set $\tilde{\mathbb C}=(\mathcal C,\otimes,I,\tilde a,c,l,r,U,t)$
is an IBMC if and only if $g\in \mathrm{GRT}_1(\bf k)$,
i.e. it satisfies \eqref{hexagons-GRT}-\eqref{pentagon-GRT}
(c.f. \cite{Dr}.)
This yields that $\mathrm{GRT}_1(\bf k)$ forms a group by 
\eqref{multiplication-GRT}.

Let $g\in \mathrm{GRT}_1(\bf k)$ and $h\in\exp{\frak t}^0_{3,N}$.
Let $\mathbb C=(\mathcal C,\otimes,I,a,c,l,r,U,t,\sigma)$ be a $C_N$-IBMC
and $\mathbb M=(\mathcal M,\otimes,b,r,V,t)$ be an IMC over it.
Define $\tilde{\mathbb C}$ as above.
Put $\tilde{\mathbb M}=(\mathcal M,\otimes,\tilde b,r,V,t)$
with 
\begin{align*}
\tilde b_{MXY}&=b_{MXY}\cdot h\Bigl(b_{MXY}^{-1}(t_{MX}\otimes id_Y)b_{MXY},
id_M\otimes t_{XY},\\
&b_{MXY}^{-1}(id_{M\otimes X}\otimes\sigma_Y)b_{MXY}(id_M\otimes t_{XY}),
\dots,\\
&\dots,
b_{MXY}^{-1}(id_{M\otimes X}\otimes\sigma^{N-1}_Y)b_{MXY}(id_M\otimes t_{XY})
\Bigr)^{-1}.
\end{align*}

\begin{lem}
The new set $\tilde{\mathbb M}$ is an IMC over
$\tilde{\mathbb C}$ if and only if $(g,h)\in \mathrm{GRTM}_{(\bar 1,1)}(N,\bf k)$,
\end{lem}

\begin{proof}
The equation \eqref{mixed pentagon-GRTM}, \eqref{octagon-GRTM} and
\eqref{special-b-GRTM}
guarantee respectively the mixed pentagon axiom in (I),
the octagon axiom in (II) and the first equality in (IV).
As for the second equality in (IV), it is automatic because
$t_{MX}\otimes id_Y$ and $(id_{M\otimes  X}\otimes\sigma^{a}_Y)b_{MXY}(id_M\otimes t_{XY})b_{MXY}^{-1}$
($a\in C_N$)
commute with $t_{M\otimes X,Y}+t_{MX}\otimes id_Y$.

Conversely by taking  $(\mathbb C,\mathbb M)=(\mathbb C_{\text{univ}},\mathbb M_{\text{univ}})$,
one sees that the presentations (I)-(IV) imply
the relations \eqref{mixed pentagon-GRTM}-\eqref{special-b-GRTM}.
\end{proof}

As a corollary we get 
\begin{prop}\label{GRTM forms a group}
The set $\mathrm{GRTM}_{(\bar 1,1)}(N,\bf k)$ forms a group by 
the multiplication \eqref{multiplication-GRTM}.
\end{prop}

For $n\geqslant 1$ define $\mathrm{GRTM}_{(\bar 1,1)}(N,{\bf k})^{(n)}$
to be the set of $(g,h)\pmod{\deg n}\in\exp{\frak t}^0_3\times\exp{\frak t}^0_{3,N}\pmod{\deg n}$ which satisfies all the defining equations of
$\mathrm{GRTM}_{(\bar 1,1)}(N,{\bf k})$ modulo $\deg n$.
By considering all IMC $\mathbb M$ over $C_N$-IBMC $\mathbb C$ such that
$\Gamma^{n+1}U_X=\Gamma^{n+1}V_M=\{1\}$
($\Gamma^n$: the $n$-th term of lower central series)
holds for any $X\in\text{Ob}\mathbb C$ and $M\in\text{Ob}\mathbb M$,
we see that $\mathrm{GRTM}_{(\bar 1,1)}(N,{\bf k})^{(n)}$ forms an algebraic group.
The following was required to prove Theorem \ref{pentagon and octagon}.

\begin{lem}\label{surjection}
The natural morphism 
$\mathrm{GRTM}_{(\bar 1,1)}(N,{\bf k})\to \mathrm{GRTM}_{(\bar 1,1)}(N,{\bf k})^{(n)}$
is surjective.
\end{lem}

\begin{proof}
This is a morphism of pro-unipotent algebraic group, which induces
a Lie algebra morphism
$$\frak{grtm}_{(\bar 1,1)}(N,{\bf k})=\prod_{k=1}^\infty
\frak{grtm}_{(\bar 1,1)}(N,{\bf k})^{(k)}\to
\frak{grtm}^{(n)}_{(\bar 1,1)}(N,{\bf k})=\oplus_{k=1}^n
\frak{grtm}_{(\bar 1,1)}(N,{\bf k})^{(k)}.
$$
Here $\frak{grtm}_{(\bar 1,1)}(N,{\bf k})^{(k)}$ means the degree $k$-component
of $\frak{grtm}_{(\bar 1,1)}(N,{\bf k})$.
Since the Lie algebra morphism is surjective, so is the pro-algebraic group
morphism.
\end{proof}

\section{Erratum of \cite{E}}\label{Erratum}
We use this opportunity to correct some errors
in \cite{E}.
\begin{enumerate}
\item
The first two formulae in \cite{E} page 400 should be replaced by
$$
\Phi_{\mathrm KZ}^{0,1,23}\Phi_{\mathrm KZ}^{01,2,3}=
\Phi_{\mathrm KZ}^{1,2,3}\Phi_{\mathrm KZ}^{0,12,3}\Phi_{\mathrm KZ}^{0,1,2}
$$
and
$$
\Psi_{\mathrm KZ}^{0,1,23}\Psi_{\mathrm KZ}^{01,2,3}=
\Phi_{\mathrm KZ}^{1,2,3}\Psi_{\mathrm KZ}^{0,12,3}\Psi_{\mathrm KZ}^{0,1,2}.
$$
\item
For a ring $R$ and $N\geqslant 1$ the definition of the ring $R(N)$ in \cite{E}\S 6.2
should be read as
follows\footnote{The first author is grateful to I. Marin for pointing out the inconsistency of
the definition in \cite{E}.}:
$R(N) = ({\bf Z}/N{\bf Z})\times R$, with the following operations. The sum is
given by
$$(a,r)+(a',r') = (a+a',r+r'+\sigma(a,a'))$$
and the product is given by
$$(a,r)(a',r') = (aa',\tilde a r' + \tilde a' r + Nrr'+\pi(a,a')),$$
where
$a\mapsto \tilde a$ is the map $({\bf Z}/N{\bf Z})\to\{0,1,\ldots,N-1\}$
inverse to the \lq reduction modulo $N$' map, and $\sigma,\pi:({\bf Z}/N{\bf Z})^2\to{\bf Z}$
are defined by $\widetilde{a+a'} = \tilde a+\tilde a'+N\sigma(a,a')$,
$\widetilde{aa'} = \tilde a\tilde a'+N\pi(a,a')$.
\end{enumerate}


\end{document}